\documentclass[reqno]{amsart}
\usepackage{amssymb}
\usepackage{amsmath}
\usepackage[mathscr]{euscript}

\usepackage{epsfig}
\usepackage[small]{caption}
\usepackage{psfrag}

\makeatletter
\@addtoreset{equation}{section}
\makeatother

\newtheorem{theorem}{Theorem}[section]
\newtheorem{lemma}[theorem]{Lemma}

\newtheorem{corollary}[theorem]{Corollary}

\newcommand{\mc}[1]{{\mathcal #1}}
\newcommand{\mf}[1]{{\mathfrak #1}}
\newcommand{\mb}[1]{{\mathbf #1}}
\newcommand{\bb}[1]{{\mathbb #1}}

\newcommand{\ms}[1]{{\mathscr #1}}

\newcommand{\<}{\langle}
\renewcommand{\>}{\rangle}
\renewcommand{\Cap}{{\rm cap}}

\begin{document}

\title[Hitting times of rare events in BDSSEP]
{Hitting times of rare events in boundary driven
  symmetric simple exclusion processes}

\author{O. Benois, C. Landim, M. Mourragui}

\address{\noindent CNRS UMR 6085, Universit\'e de Rouen, Avenue de
  l'Universit\'e, BP.12, Technop\^ole du Madril\-let, F76801
  Saint-\'Etienne-du-Rouvray, France.  \newline e-mail: \rm
  \texttt{olivier.benois@univ-rouen.fr} }

\address{\noindent IMPA, Estrada Dona Castorina 110, CEP 22460 Rio de
  Janeiro, Brasil.  \newline e-mail: \rm
  \texttt{brunog@impa.br} }

\address{\noindent IMPA, Estrada Dona Castorina 110, CEP 22460 Rio de
  Janeiro, Brasil and CNRS UMR 6085, Universit\'e de Rouen, Avenue de
  l'Universit\'e, BP.12, Technop\^ole du Madril\-let, F76801
  Saint-\'Etienne-du-Rouvray, France.  \newline e-mail: \rm
  \texttt{landim@impa.br} }

\address{\noindent CNRS UMR 6085, Universit\'e de Rouen, Avenue de
  l'Universit\'e, BP.12, Technop\^ole du Madril\-let, F76801
  Saint-\'Etienne-du-Rouvray, France.  \newline e-mail: \rm
  \texttt{mustapha.mourragui@univ-rouen.fr} }

\keywords{hitting times, rare events, non-reversible Markov processes,
interacting particle systems}

\begin{abstract}
  In the boundary driven symmetric simple exclusion process consider
  an open set $\ms O$ of density profiles which does not contain the
  stationary density profile. We prove that the first time the
  empirical measure visits the set $\ms O$ converges to an exponential
  distribution.
\end{abstract}

\maketitle


\section{Introduction}
\label{sec0}

It has long been observed that in finite-state, reversible Markov
processes the hitting time of a rare event is approximately
exponentially distributed \cite{k1, b1, ab1, ab2}. For non-reversible
dynamics much less is known. By estimating the total variation
distance between the stationary measure and the quasi-stationary
measure, Aldous \cite{a1} proved that the distribution of the hitting
time of a rare event is close to an exponential random variable when
the mixing time is small compared to the stationary expectation of the
hitting time. Fill and Lyzinski \cite{fl1} proved that starting from
the stationary distribution the hitting time of a configuration $\eta$
can be represented as an independent geometric sum of i.i.d. random
variables if the probability of hitting this configuration $\eta$ at
time $t$ starting from $\eta$, viz. $p_t(\eta,\eta)$, decreases in
time. This representation permits to obtain bounds for the distance
between the distribution of the hitting time and the distribution of
an exponential random variable. Imbuzeiro \cite{i1} proved that the
hitting time of a rare event $A$ is approximately exponential starting
from a distribution $\nu$ if starting from $\nu$ the probability of
hitting $A$ before the mixing time is small.  Fernandez et
al. \cite{fmns1} are presently working on this problem in the sequel
of \cite{bbf1}.

In this article we examine the hitting time of rare events in a well
studied non-reversible dynamics, the boundary driven symmetric simple
exclusion processes (BDSSEP). Beyond the complications arising from
non-reversibility, this model presents a further difficulty in the
lack of an explicit formula for the stationary measure. This obstacle
is overcome by the use of a large deviations principle to estimate the
measure of sets, but prevents us from obtaining bounds for the
stationary expectation of the hitting time with errors sharper than
exponential.

In the context of interacting particle systems the convergence of
hitting times of rare events to exponential random variables has been
abundantly investigated. Several results have been obtained for
non-conservative dynamics, processes in which the local number of
particles changes in time and which lose memory much faster than
conservative ones. On the conservative side, which includes the the
dynamics examined here, Ferrari et al. \cite{fgl1} considered the case
of a totally asymmetric one-dimensional zero-range process, and
Ferrari et al. \cite{fgl2} examined the case of the one-dimensional
symmetric simple exclusion process. This latter result was generalized
to any dimension and extended to independent random walks by Asselah
and Dai Pra \cite{ap2, ap3}.

The article is organized as follows. In the next section we state the
main result. In Section \ref{sec2} we present a general method to
derive the asymptotic exponentiality of the hitting time of a rare
event for finite-state, non-reversible continuous-time Markov
processes starting from a measure not too far from the stationary
measure in the sense of Lemmas \ref{s17}, \ref{s14} or \ref{s15}. In
Section \ref{sec3} we estimate the expectation of the hitting time
under the stationary state assuming a dynamical large deviations
principle. In Section \ref{sec4} we apply the results presented in the
two previous section to the BDSSEP.

\section{Notation and Results}
\label{sec1}

\noindent{\bf The one-dimensional boundary driven symmetric simple
  exclusion process (BDSSEP)}. For $N\ge 1$, let $\Lambda_N =\{1,
\dots, N-1\}$.  Fix $0< \alpha\le \beta< 1$ and consider the Markov
process $\{\eta^N(t) : t\ge 0\}$ on $\Omega_N = \{0,1\}^{\Lambda_N}$
whose generator $L_N$ is given by
\begin{eqnarray*}
(L_N f)(\eta) &=& \frac 12 \, \sum_{x=1}^{N-2} \{ f(\sigma^{x,x+1}
\eta)-f(\eta)\} \\
&+& \frac 12 \, \Big\{ \alpha [1-\eta(1)] + (1-\alpha) \eta(1) \Big\}
\{ f(\sigma^{1} \eta)-f(\eta)\} \\
&+& \frac 12 \, \Big\{ \beta [1-\eta(N-1)] + (1-\beta) \eta(N-1) \Big\}
\{ f(\sigma^{N-1} \eta)-f(\eta)\}\;.
\end{eqnarray*}
In this formula, $\eta = \{\eta(x),\, x\in\Lambda_N\}$ is a
configuration of the state space $\{0,1\}^{\Lambda_N}$ so that $\eta
(x)=0$ if and only if site $x$ is vacant for $\eta$;
$\sigma^{x,y}\eta$ is the configuration obtained from $\eta$ by
interchanging the occupation variables $\eta(x)$, $\eta(y)$:
$$
(\sigma^{x,y} \eta) (z)\; =\;
\left\{
\begin{array}{ll}
\eta (z)  & \hbox{if $z\neq x,y$}\; , \\
\eta (y)  & \hbox{if $ z=x$}\; , \\
\eta (x)  & \hbox{if $ z=y$}\; ;
\end{array}
\right.
$$
and $\sigma^{x}\eta$ is the configuration obtained from $\eta$ by
flipping the variable $\eta(x)$:
$$
(\sigma^{x} \eta) (z)\; =\;
\left\{
\begin{array}{ll}
\eta (z)  & \hbox{if $z\neq x$}\; , \\
1-\eta (z)  & \hbox{if $ z=x$}\; .
\end{array}
\right.
$$
Hence, at rate $\alpha$ (resp. $1-\alpha$) a particle is created
(resp. removed) at the boundary site $1$ if this site is vacant (resp.
occupied). The same phenomenon occurs at the boundary $x=N-1$ with
$\beta$ in place of $\alpha$.

Denote by $D(\bb R_+, \Omega_N)$ the Skorohod space of paths from $\bb
R_+$ to $\Omega_N$.  Let $\bb P^N_\eta$, $\eta\in\Omega_N$, be the
distribution of the Markov process $\eta^N(t)$ when the initial
configuration is $\eta$. The probability measure $\bb P^N_\eta$ is
thus a measure on the path space $D(\bb R_+, \Omega_N)$ endowed with
the Skorohod topology. Expectation with respect to $\bb P^N_\eta$ is
denoted by $\bb E^N_\eta$.

The finite state Markov process $\eta^N(t)$ is irreducible and has
therefore a unique stationary measure, denoted by $\nu^N_{\alpha,
  \beta}$. The process is reversible if and only if $\alpha=\beta$, in
which case the measure $\nu^N_{\alpha, \alpha}$ is a product measure.

\smallskip{\bf The empirical measure.} Denote by $\langle \cdot,\cdot
\rangle$ the inner product in $L_2 ([0,1])$ and set
\begin{equation*}
\label{dcm}
\ms M := \left\{ \rho \in L_\infty ([0,1]) \,:\:
0\leq \rho \leq 1 \right\}
\end{equation*}
which we equip with the topology induced by the weak convergence of
measures, namely a sequence $\{\rho^n : n\ge 1\} \subset \ms M$
converges to $\rho$ in $\ms M$ if and only if $\langle \rho^n, G
\rangle \to \langle \rho, G \rangle$ for any continuous function $G:
[0,1]\to\bb R$.  Note that $\ms M$ is a compact Polish space that we
consider endowed with the corresponding Borel $\sigma$-algebra.  

Let $d$ be a distance in $\ms M$ compatible with the weak topology,
\begin{equation}
\label{24}
d(\gamma,\gamma') \;=\; \sum_{k\ge 1} \frac 1{2^k} \, \big| \< \gamma,
F_k\> - \< \gamma', F_k\> \big|\;,
\end{equation}
where the continuous test functions $F_k$ are absolutely bounded by
$1$.

The empirical density of a configuration $\eta\in \Omega_N$, denoted
by $\pi^N(\eta) \in \ms M$, is defined as
\begin{equation*}
\label{eq:2}
\pi^N (\eta) \;:=\; \sum_{x=1}^{N-1} \eta (x) \,
\mb 1 \big\{ 
{\textstyle \big[ \frac{x}{N}- \frac{1}{2N}, \frac{x}{N}
+ \frac{1}{2N}\big)} \big\}\; , 
\end{equation*}
where $\mb 1\{A\}$ stands for the indicator function of the set $A$.

Denote by $\nabla$ the space derivative and by $\Delta$ the Laplacian.
It has been proved in \cite{els1} that under the stationary state
$\nu^N_{\alpha, \beta}$ the empirical measure $\pi^N$ converges in
probability to the unique solution of the elliptic equation
$$
\left\{
\begin{array}{l}
\Delta \rho = 0\;, \\
\rho(0) = \alpha\;, \quad \rho(1) = \beta \;.
\end{array}
\right.
$$
We denote the solution of this equation by $\bar\rho =
\bar\rho_{\alpha, \beta}$.

\subsection*{The dynamical rate function} To state the main result of
this article we need to introduce the rate functions of the dynamical
and the static large deviations principle of the empirical measure. We
start with the dynamical one.

For $T>0$ and positive integers $m,n$, we denote by
$C^{m,n}([0,T]\times[0,1])$ the space of functions $G \colon
[0,T]\times[0,1] \to\bb R$ with $m$ derivatives in time, $n$
derivatives in space which are continuous up to the boundary. We
improperly denote by $C^{m,n}_0([0,T]\times[0,1])$ the subset of
$C^{m,n}([0,T]\times[0,1])$ of the functions which vanish at the
endpoints of $[0,1]$, i.e.\ $G\in C^{m,n}([0,T]\times[0,1])$ belongs
to $C^{m,n}_0([0,T]\times[0,1])$ if and only if $G(t, 0)=G(t, 1)=0$,
$t\in[0,T]$.

Let the energy $\mc Q: D([0,T], \ms M) \to [0,\infty]$ be given by
\begin{eqnarray*}
\!\!\!\!\!\!\!\!\!\!\!\!\! &&
\mc Q(u) \;=\;  \\
\!\!\!\!\!\!\!\!\!\!\!\!\! && \quad
\sup_{G} \Big\{ \int_0^T dt \int_{0}^1 d\mb x\,  u(t,\mb x) \,
(\nabla G)(t,\mb x) \;-\; \frac 12  \int_0^T dt \int_{0}^1 d\mb x\,
G(t,\mb x)^2 \, \chi (u(t,\mb x)) \Big\}\;, 
\end{eqnarray*}
where $\chi:[0,1]\to\bb R_+$ is the mobility of the system, $\chi(a) =
a(1-a)$, and where the supremum is carried over all smooth functions
$G: [0,T]\times (0,1) \to\bb R$ with compact support.  It has been
shown in \cite{blm1} that the energy $\mc Q$ is convex and lower
semicontinuous. Moreover, if $\mc Q(u)$ is finite, $u$ has a
generalized space derivative, $\nabla u$, and
\begin{equation*}
\mc Q(u) \;=\; \frac 12 \int_0^T dt\, \int_{0}^1 d\mb x\,
\frac{(\nabla u(t))^2}{\chi (u(t))}\;\cdot 
\end{equation*}

Fix a function $\gamma\in\ms M$ which corresponds to the initial
profile.  For each $H$ in $C^{1,2}_0([0,T]\times[0,1])$, let $\hat
J_{H}(.|\gamma) = \hat J_{T, H, \gamma} \colon D([0,T], \ms M)\longrightarrow
\bb R$ be the functional given by
\begin{eqnarray}
\label{o01}
\hat J_H(u|\gamma) &:=& \big\langle u_T, H_T \big\rangle 
- \langle \gamma, H_0 \rangle
- \int_0^{T} \!dt\, \big\langle u_t, \partial_t H_t \big\rangle
\;-\; \frac 12 \int_0^{T} \!dt\, \big\langle u_t , \Delta H_t
\big\rangle\\
&+ &\frac {\beta}2   \int_0^{T} \! dt\, \nabla H_t(1) 
\; -\;  \frac {\alpha}2  \int_0^{T} \!dt\, \nabla H_t(0) 
\;-\; \frac 12 \int_0^{T} \!dt\,
\big\langle \chi( u_t ), \big( \nabla H_t \big)^2 \big\rangle \; .    
\nonumber 
\end{eqnarray}
Let $\hat I_{[0,T]} (\, \cdot \, | \gamma) \colon D([0,T],\mc
M)\longrightarrow [0,+\infty]$ be the functional defined by
\begin{equation*}
\hat I_{[0,T]} (u | \gamma) \; :=\; \sup_{H\in
  C^{1,2}_0([0,T]\times[0,1])} \hat J_H(u|\gamma)\; .
\end{equation*}
The dynamical rate functional $I_{[0,T]}(\cdot | \gamma): D([0,T], \ms
M) \to [0,\infty]$ is given by
\begin{equation}
\label{26}
I_{[0,T]}(u | \gamma) \;=\; 
\left\{
\begin{array}{ll}
\hat I_{[0,T]} (u | \gamma) & \text{if $\mc Q(u) < \infty$ \;,} \\
\infty & \text{otherwise.}
\end{array}
\right.
\end{equation}

\subsection*{The static rate functional}
Denote by $V: \ms M \to \bb R_+$ the quasi-potential associated to the
dynamical rate functions $I_{[0,T]}$:
\begin{equation}
\label{08}
V(\gamma) \;=\; \inf_{T>0} \inf \big\{ I_{[-T,0]} (u | \bar\rho) : u(-T) =
\bar\rho \,,\, u(0)=\gamma \,\}\;.
\end{equation}

It has been proved in \cite[Theorems 2.2, 4.5 and A.1]{bdgjl3} that
$V$ is bounded, convex and lower-semicontinuous, and that $V(\rho) >0$
for all $\rho\not = \bar\rho$. We are now in a position to state the
main result of this article.

\begin{theorem}
\label{s16}
Fix an open subset $\ms O$ of $\ms M$ such that $d(\bar\rho, \ms O)>0$
and let $H_{\ms O}$ be the hitting time of the set $\ms O$, $H_{\ms
  O}= \inf\{t : \pi^N(\eta(t)) \in \ms O\}$. Then, under
$\nu^N_{\alpha, \beta}$, $H_{\ms O}/E_{\nu^N_{\alpha, \beta}} [H_{\ms
  O}]$ converges in distribution to a mean one exponential
time. Moreover, if 
\begin{equation*}
\inf_{\gamma\in\ms O} V(\gamma) \;=\; 
\inf_{\gamma\in\overline{\ms O}} V(\gamma)\;,
\end{equation*}
where $\overline{\ms O}$ represents the closure of $\ms O$, we have
that
\begin{equation*}
\lim_{N\to\infty} \frac 1N \log 
\bb E_{\nu^N_{\alpha, \beta}}[H_{\ms O}] \;=\; \inf_{\gamma\in\ms O} V(\gamma)\;.
\end{equation*}
Finally, consider a subset $\ms B$ of $\ms M$ such that
\begin{equation}
\label{23}
\inf_{\gamma\in \ms B^o} V(\gamma) \;<\; 
\inf_{\gamma\in \overline{\ms O}} V(\gamma)\,
\end{equation} 
where $\ms B^o$ stands for the interior of $\ms B$.  Let $B_N =
(\pi^N)^{-1}(\ms B) = \{\eta\in\Omega_N : \pi^N(\eta)\in \ms B\}$ and
let $\mu_N$ be the probability measure on $\Omega_N$ defined by
$\mu_N(\eta) = \mb 1\{\eta\in B_N\} \nu^N_{\alpha,
  \beta}(\eta)/\nu^N_{\alpha, \beta}(B_N)$. Then, under $\mu_N$,
$H_{\ms O}/E_{\nu^N_{\alpha, \beta}} [H_{\ms O}]$ converges in
distribution to a mean one exponential time.
\end{theorem}

This result holds in all dimensions, we restricted ourselves to
dimension one for sake of simplicity.

\section{Hitting times of rare events have exponential distributions}
\label{sec2}

Consider a sequence of irreducible, continuous-time Markov processes
$\{\eta^N(t) : t\ge 0\}$, $N\ge 1$, taking values on a finite state
space $\Omega_N$. The points of $\Omega_N$ are represented by the
Greek letters $\eta$, $\xi$. Denote by $\nu_N$ the unique stationary
state, by $L_N$ the generator of the process, by $\lambda_N(\eta)$,
$\eta\in \Omega_N$, the holding rates, by $p_N(\eta,\xi)$,
$\xi\not=\eta\in \Omega_N$, the jump probabilities, and by
$R_N(\eta,\xi) = \lambda _N(\eta)\, p_N(\eta,\xi)$ the jump rates. In
particular, for every function $f:\Omega_N\to \bb R$,
\begin{equation*}
(L_N f)(\eta) \;=\; \sum_{\xi\in\Omega_N} R_N(\eta,\xi) \,[f(\xi) -
f(\eta)]\; ,
\end{equation*}
We often omit the superscript $N$ of $\eta^N(t)$.

For a subset $A$ of $\Omega_N$, denote by $H_A$ (resp. $H^+_A$) the
hitting (resp. return) time of a set $A$:
\begin{equation*}
\begin{split}
& H_A \,:=\, \inf \big\{ s > 0 : \eta (s) \in A \big\}\;, \\
& \quad H^+_A \,:=\, \inf \{ t>0 : \eta (t) \in A, \eta(s) \not= \eta(0)
\;\;\textrm{for some $0< s < t$}\}\;.
\end{split}
\end{equation*}
When the set $A$ is a singleton $\{\eta\}$, we denote $H_{\{\eta\}}$,
$H^+_{\{\eta\}}$ by $H_\eta$, $H^+_{\eta}$, respectively.

Let $D(\bb R_+, \Omega_N)$ be the space of $\Omega_N$-valued, right
continuous paths with left limits endowed with the Skorohod topology.
Denote by $\bb P_\eta = \bb P^N_\eta$, $\eta\in\Omega_N$, the
probability measure on $D(\bb R_+, \Omega_N)$ induced by the Markov
process $\eta(t)$ starting from $\eta$. Expectation with respect to
$\bb P_\eta$ is represented by $\bb E_\eta$. For a probability measure
$\mu$ in $\Omega_N$, $\bb P_\mu[\,\cdot\,] = \sum_{\eta\in\Omega_N}
\mu(\eta) \, \bb P_\eta [\,\cdot\,]$, with the same notation for
expectations.

Let $P_t (\eta,\xi)$, $t\ge 0$, $\eta$, $\xi\in\Omega_N$, be the
semigroup associated to $\eta(t)$, $P_t (\eta,\xi) = \bb P_\eta [
\eta(t)=\xi]$. Denote by $\Vert\mu -\nu\Vert_{\rm TV}$ the total
variation distance between two probability measures $\mu$ and $\nu$
defined on $\Omega_N$. Let $T^{\rm mix}_N$ be the mixing time of the
process $\eta(t)$:
\begin{equation*}
T^{\rm mix}_N \;=\; \inf\Big\{t>0 : \max_{\eta\in\Omega_N} \Vert
P_t(\eta, \,\cdot\,) - \nu_N \Vert_{\rm TV} \le \frac 14 \Big\}\;.
\end{equation*}

Let $A_N$ be a sequence of subsets of $\Omega_N$ such that
\begin{equation}
\label{05}
\lim_{N\to\infty} \nu_N(A_N) \;=\; 0\;.
\end{equation}
Denote by $H_N = H_{A_N}$ the hitting time of $A_N$:
\begin{equation*}
H_N \;=\; \inf\{t>0: \eta^N(t)\in A_N\}\;,
\end{equation*}
and by $r_N(A_N^c, A_N)$ the average rate at which the process jumps
from $A_N^c $ to $A_N$:
\begin{equation*}
r_N(A_N^c, A_N) \;=\; \frac 1{\nu_N (A_N^c)} 
\sum_{\xi \in A_N^c} \nu_N (\xi)\, R_N(\xi, A_N)\;,
\end{equation*}
where $R_N(\xi, A_N) = \sum_{\zeta\in A_N} R_N(\xi, \zeta)$.

Nest statement is the main result of this section. It has to be
compared with \cite[Theorem 1.4]{a1}. Instead of requiring that the
mixing time is small compared to the the stationary expectation of the
hitting time, we assume that the mixing time is small compared to the
inverse of the averaged jump rate, $r_N(A_N^c, A_N)^{-1}$, a quantity
easily estimated. Moreover, by \cite[Lemma 2.3]{bg1}, for reversible
dynamics, $r_N(A_N^c, A_N)^{-1}$ is bounded by the expected value of
the hitting time of $A_N$ starting from the quasi-stationary state.

\begin{theorem}
\label{s04}
Let $A_N$ be a sequence of subsets of $\Omega_N$ satisfying
\eqref{05}.  Assume that $T^{\rm mix}_N \ll r_N(A_N^c,
A_N)^{-1}$. Then, under $\nu_N$, the sequence $H_N/\bb E_{\nu_N}
[H_N]$ converges in distribution to a mean one exponential random
variable.
\end{theorem}

Theorem \ref{s04} follows from Lemmas \ref{s03} and \ref{s02}.  We
prove below in \eqref{06} and Lemma \ref{s02} that
\begin{equation*}
\liminf_N r_N(A_N^c, A_N) \, \bb E_{\nu_N} [H_N]  \;>\; 0\;.
\end{equation*}
In fact, under some assumptions this product converges to $1$.
To state this hypotheses we need to introduce some notation.

For two disjoint subsets $A$, $B$ of the state space $\Omega_N$,
denote by $\Cap (A, B)$ the capacity between $A$ and $B$:
\begin{equation*}
\Cap (A, B) \;=\; \sum_{\eta\in A} \nu_N(\eta) \, \lambda_N(\eta)\, 
\bb P_{\eta} \big[ H_B < H^+_{A} \big]\;.
\end{equation*}
When the set $A$ is a singleton, $A=\{\eta\}$, we write $\Cap(\eta,
B)$ for $\Cap (\{\eta\}, B)$.

Denote by $\{\eta^* (t) : t\ge 0\}$ the stationary Markov process
$\eta(t)$ reversed in time. We shall refer to $\eta^*(t)$ as the
adjoint or the time reversed process. It is well known that
$\eta^*(t)$ is a Markov process on $\Omega_N$ whose generator $L^*_N$
is the adjoint of $L_N$ in $L^2(\nu_N)$.  The jump rates
$R^*_N(\eta,\xi)$, $\eta\not=\xi\in \Omega_N$, of the adjoint process
satisfy the balanced equations
\begin{equation*}
\nu_N(\eta) \, R_N(\eta,\xi) \;=\; \nu_N(\xi) \, R^*_N(\xi,\eta)\;.
\end{equation*}
Denote by $\lambda^*(\eta)=\lambda(\eta)$, $\eta\in \Omega_N$,
$p^*(\eta,\xi)$, $\eta\not=\xi\in \Omega_N$, the holding rates and the
jump probabilities of the time reversed process $\eta^*(t)$.

As above, for each $\eta\in \Omega_N$, denote by $\bb P^*_\eta$ the
probability measure on the path space $D(\bb R_+, \Omega_N)$ induced
by the Markov process $\eta^*(t)$ starting from $\eta$. Expectation
with respect to $\bb P^*_\eta$ is denoted by $\bb E^*_\eta$.

\begin{lemma}
\label{s05}
Assume that there exists a sequence of subsets $B_N$, $B_N\subset
A^c_N$, $\lim_{N\to\infty} \nu_N(B_N) = 1$, such that
\begin{equation}
\label{07}
\lim_{N\to\infty} \sup_{\substack{\eta\in B_N \\ \xi\not\in A_N}} 
\bb P^*_{\xi} [ H_{A_N}  < H_\eta ] \;=\; 0\;,
\quad \lim_{N\to\infty} \sum_{\eta\not\in A_N\cup B_N} 
\frac{\nu_N(\eta)} {\Cap (\eta, A_N)}\;=\;0\;.
\end{equation}
Assume, furthermore, that
\begin{equation}
\label{25}
\limsup_{N\to\infty} r_N(A_N^c, A_N) \;<\;\infty\;.
\end{equation}
Then,
\begin{equation*}
\lim_{N\to\infty} r_N(A_N^c, A_N)\, \bb E_{\nu_N} [H_N]  
\;=\; 1\;.  
\end{equation*}
\end{lemma}

\begin{proof}
Fix $\eta\not\in A_N$.  By definition of the capacity, by equation
(2.4) and Lemma 2.3 in \cite{gl2}, and by the Markov property,
\begin{equation*}
\begin{split}
\Cap (\eta, A_N) \;&=\; \Cap^* (\eta, A_N) \;=\; 
\sum_{\substack{\xi\in A_N \\ \zeta\not\in
    A_N}} \nu_N(\xi)\, R^*_N(\xi,\zeta) \, \bb P^*_{\zeta}[H_\eta <
H_{A_N}] \\
&=\; \sum_{\substack{\xi\in A_N \\ \zeta\not\in
    A_N}} \nu_N(\zeta)\, R_N(\zeta, \xi) \, \bb P^*_{\zeta}[H_\eta <
H_{A_N}] \;. 
\end{split}
\end{equation*}
This sum is bounded above by $\nu_N(A_N^c) r_N(A_N^c, A_N) \le
r_N(A_N^c, A_N)$. On the other hand, if $\eta$ belongs to $B_N$, by
assumption \eqref{07}, the sum is bounded below by $(1-\epsilon_N)
\nu_N(A_N^c) r_N(A_N^c, A_N) \ge (1-\epsilon_N) r_N(A_N^c, A_N)$,
where $\epsilon_N$ is a sequence which vanishes as $N\uparrow\infty$,
and which may change from line to line.

By \cite[Proposition A.2]{bl7}, 
\begin{equation*}
\bb E_{\nu_N} [H_N] 
\;=\; \sum_{\eta\not\in A_N} \nu_N(\eta) \, \bb E_{\eta} [H_N]
\;=\; \sum_{\eta\not\in A_N} \nu_N(\eta) \,
\frac{\sum_{\xi\not\in A_N} \nu_N(\xi) \, \bb P^*_{\xi} [ H_\eta < H_{A_N}]}
{\Cap (\eta, A_N)}\;\cdot
\end{equation*}

By the lower bound for the capacity obtained in the beginning of the
proof and by \eqref{07}, this expression is bounded above by
\begin{equation*}
\sum_{\eta\in B_N} \frac{\nu_N(\eta)} {\Cap (\eta, A_N)}
\;+\; \sum_{\eta\not\in A_N\cup B_N} \frac{\nu_N(\eta)} {\Cap (\eta,
  A_N)} \;\le\; (1+\epsilon_N) \, r_N(A_N^c, A_N)^{-1}
\;+\; \epsilon_N\;.
\end{equation*}
In view of \eqref{25}, this proves that
\begin{equation*}
\limsup_{N\to\infty} r_N(A_N^c, A_N)\, \bb E_{\nu_N} [H_N]  \;\le\; 1\;.
\end{equation*}

By \eqref{07} and by the upper bound for the capacity obtained in the
beginning of the proof, $\bb E_{\nu_N} [H_N]$ is bounded below by
\begin{equation*}
\begin{split}
& \sum_{\eta\in B_N} \nu_N(\eta)
\frac{\sum_{\xi\not\in A_N} \nu_N(\xi) \, \bb P^*_{\xi} 
[ H_\eta < H_{A_N}]} {\Cap (\eta, A_N)} \\
&\quad \ge\;
(1-\epsilon_N) \sum_{\eta\in B_N} 
\frac{\nu_N(\eta)} {\Cap (\eta, A_N)}\;\ge\;
(1-\epsilon_N) \,  r_N(A_N^c, A_N)^{-1} \;.  
\end{split}
\end{equation*}
This concludes the proof of the lemma.
\end{proof}

Denote by $N_t$, $t\ge 0$, the number of jumps from $A_N^c$ to
$A_N$ in the time interval $[0,t]$. $N_t$ is a Poisson process and
$M_t$, defined by
\begin{equation*}
M_t\;=\; N_t \;-\; \int_0^t R_N(\eta(s), A_N) \, \mb 1\{\eta(s)
\not\in A_N\} \, ds\;,
\end{equation*}
is a martingale. In particular,
\begin{equation*}
\bb E_{\nu_N} [N_t] \;=\; t \, \nu_N
(A_N^c)\, r_N(A_N^c, A_N) \;.
\end{equation*}

Note that $\{H_N\le t\} = \{\eta(0)\in A_N\} \cup \{N_t \ge 1\}$.
Define
\begin{equation*}
X_t \;=\; \mb 1 \{\eta(0)\in A_N\} \;+\; N_t
\end{equation*}
so that $\{H_N\le t\} = \{X_t \ge 1\}$, and 
\begin{equation}
\label{01}
\begin{split}
\bb P_{\nu_N} [H_N \le t] \; &=\;
\bb P_{\nu_N} [X_t \ge 1] \;\le\;
\bb E_{\nu_N} [X_t]  \\
& \le\; \nu_N (A_N) \;+\; t\, \nu_N
(A_N^c)\, r_N(A_N^c, A_N)\;.
\end{split}
\end{equation} 

\begin{lemma}
\label{s01}
Assume that $T^{\rm mix}_N \ll r_N(A_N^c, A_N)^{-1}$. Let $\gamma_N$,
$\sigma_N$ be two sequences such that $T^{\rm mix}_N \ll \sigma_N \ll
\min\{\gamma_N, r_N(A_N^c, A_N)^{-1}\}$. Then, for every $t$, $s> 0$,
\begin{equation*}
\begin{split}
& \Big|\, \bb P_{\nu_N} \big[H_N> (t+s) \gamma_N \big] \,-\, 
\bb P_{\nu_N} \big[H_N> s \gamma_N \big] \, 
\bb P_{\nu_N} \big[H_N> t \gamma_N \big] \, \Big| \\
&\qquad \;\le\; 2\, \nu_N (A_N) \;+\; 
2\, \sigma_N \, r_N(A_N^c, A_N) \;+\;
(1/2)^{\sigma_N/T^{\rm mix}_N}\;.
\end{split}
\end{equation*}
\end{lemma}

\begin{proof}
In view of the definition of $X_t$, we have to estimate the difference 
\begin{equation*}
\bb P_{\nu_N} \big[ X_{(t+s) \gamma_N}=0 \big] \,-\, 
\bb P_{\nu_N} \big[ X_{s \gamma_N}=0 \big] \, 
\bb P_{\nu_N} \big[ X_{t \gamma_N}=0 \big] \;.
\end{equation*}

Let $\sigma_N$ be a sequence such that $T^{\rm mix}_N \ll \sigma_N
\ll r_N(A_N^c, A_N)^{-1}$. Clearly,
\begin{equation*}
\Big|\, \bb P_{\nu_N} \big[ X_{s \gamma_N}=0 \big] \,-\, 
\bb P_{\nu_N} \big[ X_{s \gamma_N} - X_{\sigma_N} =0 
\big]  \, \Big| \;\le \; 
\bb P_{\nu_N} \big[ X_{\sigma_N} \ge 1 \big] 
\;,
\end{equation*}
and, by \eqref{01}, this last probability is bounded by $\nu_N (A_N)
\;+\; \sigma_N \, r_N(A_N^c, A_N)$.  By stationarity, a similar bound
holds for the absolute value of the difference
\begin{equation*}
\bb P_{\nu_N} \big[ X_{(t+s) \gamma_N}=0 \big] \,-\,
\bb P_{\nu_N} \big[ X_{t \gamma_N}=0 \,,\,
X_{(t+s) \gamma_N} - X_{t \gamma_N + \sigma_N} =0  \big] \;.
\end{equation*}

It remains to estimate the absolute value of the difference
\begin{equation*}
\begin{split}
& 
\bb P_{\nu_N} \big[ X_{t \gamma_N}=0 \,,\,
X_{(t+s) \gamma_N} - X_{t \gamma_N + \sigma_N} =0  \big] \\
&\qquad \qquad \qquad \,-\, 
\bb P_{\nu_N} \big[ X_{s \gamma_N} - X_{\sigma_N} =0 \big] \, 
\bb P_{\nu_N} \big[ X_{t \gamma_N}=0 \big] \;.
\end{split}
\end{equation*}
By the Markov property, this expression is equal to
\begin{equation*}
\bb E_{\nu_N} \Big[ \mb 1\{ X_{t \gamma_N}=0\} \,
\Big\{ \bb P_{\eta(t \gamma_N)} \big[  X_{s \gamma_N} - X_{\sigma_N}
=0  \big] \,-\, 
\bb P_{\nu_N} \big[ X_{s \gamma_N} - X_{\sigma_N} =0
\big] \Big\}\, \Big]\;.
\end{equation*}
This expectation is absolutely bounded by
\begin{equation*}
\sup_{\eta\in \Omega_N}
\Big |\, \bb E_{\eta} \Big[  \,\bb P_{\eta(\sigma_N)} \big[ X_{s
  \gamma_N - \sigma_N} =0  \big]\, \Big] \,-\, 
\bb P_{\nu_N} \big[ X_{s\gamma_N - \sigma_N} =0
\big] \, \Big| \;\le\; (1/2)^{\sigma_N/T^{\rm mix}_N}\;.
\end{equation*}
where we used the definition of the mixing time in the last
inequality. This concludes the proof of the lemma.
\end{proof}

Let $\theta_N$ be given by
\begin{equation}
\label{02}
\theta_N \;=\; \inf \Big\{t>0 \,:\,
\bb P_{\nu_N} \big[ H_N> t \big] 
\;<\; e^{-1} \Big\}\;.
\end{equation}
Note that $\bb P_{\nu_N} \big[ H_N> \theta_N \big] 
\,\le\, e^{-1}$. Hence, by \eqref{01},
\begin{equation*}
1\,-\, e^{-1} \;\le\; \bb P_{\nu_N} \big[ H_N \le
\theta_N \big] \;\le\; \nu_N (A_N) \;+\; 
\theta_N\, \nu_N (A_N^c)\, r_N(A_N^c, A_N)\;.
\end{equation*}
Since $\nu_N (A_N)$ vanishes, we deduce from this
inequality that
\begin{equation}
\label{06}
\liminf_N \theta_N\, r_N(A_N^c, A_N) \;>\; 0\;.
\end{equation}
In particular, $\theta_N \gg T^{\rm mix}_N$. 

\begin{lemma}
\label{s03}
Assume that $T^{\rm mix}_N \ll r_N(A_N^c, A_N)^{-1}$.
Let $\theta_N$ be the sequence defined by \eqref{02}. Under
$\nu_N$, the sequence of random variables
$H_N/\theta_N$ converges in distribution to a mean one exponential
random variable.
\end{lemma}

\begin{proof} 
Since $T^{\rm mix}_N \ll r_N(A_N^c, A_N)^{-1}$, by \eqref{06}, $T^{\rm
  mix}_N \ll \theta_N$.  By Lemma \ref{s01} with $\gamma_N = \theta_N$
and some sequence $\sigma_N$, $T^{\rm mix}_N \ll \sigma_N \ll
r_N(A_N^c, A_N)^{-1}$, we have that
\begin{equation*}
\lim_{N\to\infty} \bb P_{\nu_N} \Big[\,
\frac{H_N}{\theta_N} > t \,\Big] \;=\; e^{-t}\;, \quad t>0\;.
\end{equation*}
\end{proof}

\begin{lemma}
\label{s02}
The sequence $\theta_N$ introduced in \eqref{02} satisfies
\begin{equation*}
\lim_{N\to\infty} \frac{\bb E_{\nu_N} [\,
  H_N\,]}{\theta_N} \;=\; 1\;.
\end{equation*}
\end{lemma}

\begin{proof}
Let
\begin{equation*}
\theta_N(\eta) \;:=\; \inf\big\{ t>0 \,:\, \bb P_{\eta} [\,
  H_N > t\,] \,\le\, e^{-1} \big\}\;,\quad \eta\in \Omega_N\;,
\end{equation*}
and let $\hat \theta_N = \max_{\eta\in \Omega_N} \theta_N(\eta)$. We
first claim that
\begin{equation}
\label{04}
\lim_{N\to\infty} \theta_N/\hat\theta_N \;=\; 1\;.
\end{equation}
It is clear that $\theta_N \le \hat\theta_N$. Indeed, if
$t>\hat\theta_N$, $t>\theta_N(\eta)$ for all $\eta\in
\Omega_N$, so that
\begin{equation*}
\bb P_{\nu_N} [\,H_N > t \,] \;=\;
\sum_{\eta\in \Omega_N} \nu_N (\eta)\,
\bb P_{\eta} [\, H_N > t\,] \,\le\, e^{-1}\;.
\end{equation*}
Hence, $\theta_N \le t$ and $\theta_N \le \hat\theta_N$.

To prove the converse inequality, let $\theta_N(a)$, $a>0$, be given
by 
\begin{equation*}
\theta_N(a) \;:=\; \inf\big\{ t>0 \,:\, \bb P_{\nu_N} [\,
  H_N > t\,] \,\le\, e^{-a} \big\}\;. 
\end{equation*}
For any $\eta\in\Omega_N$, $\epsilon>0$, $L\ge 1$,
\begin{equation*}
\bb P_{\eta} \big[\, H_N > \theta_N(1+\epsilon) + L T^{\rm mix}_N\,\big]
\,\le\, \bb E_{\eta} \Big[\, 
\bb P_{\eta (L T^{\rm mix}_N)} \big[\, H_N > \theta_N(1+\epsilon) \,\big]
\,\Big]
\end{equation*} 
By definition of the mixing time and of $\theta_N(1+\epsilon)$, the
last expectation is bounded by
\begin{equation*}
2^{-L}\;+\; \bb P_{\nu_N} \big[\, H_N > 
\theta_N(1+\epsilon) \,\big] \;\le\; 2^{-L}\;+\;
e^{-(1+\epsilon)}\;\le\; e^{-1} 
\end{equation*}
provided $2^{-L} \le e^{-1} [ 1 - e^{-\epsilon}]$. Hence,
$\theta_N(\eta) \le \theta_N(1+\epsilon) + L T^{\rm mix}_N$ for all
$\eta\in\Omega_N$ so that $\hat \theta_N \le \theta_N(1+\epsilon) + L
T^{\rm mix}_N$.

Denote by $R_N$ the right hand side of the inequality appearing in the
statement of Lemma \ref{s01} with $\gamma_N$ replaced by
$\theta_N$. Iterating $k-1$ times this estimate, we obtain that
\begin{equation*}
\bb P_{\nu_N} \big[\, H_N > \theta_N/k \,\big] \;\le\;
\Big(\bb P_{\nu_N} \big[\, H_N > \theta_N \,\big] +
k \, R_N\Big)^{1/k}\;.
\end{equation*}
Applying once more Lemma \ref{s01}, we get that
\begin{equation*}
\bb P_{\nu_N} \big[\, H_N > (k+1)\theta_N/k \,\big] \;\le\;
\bb P_{\nu_N} \big[\, H_N > \theta_N \,\big] \,
\bb P_{\nu_N} \big[\, H_N > \theta_N/k \,\big] +
R_N\;,
\end{equation*}
so that
\begin{equation*}
\bb P_{\nu_N} \big[\, H_N > (k+1)\theta_N/k \,\big]
\;\le\; e^{-1} \big( e^{-1} + k \, R_N\big)^{1/k} \;+\; R_N\;.
\end{equation*}
Since $R_N$ vanishes, if $k> \epsilon^{-1}$ this expression is bounded
by $e^{-(1+\epsilon)}$ for $N$ sufficiently large. Therefore,
$\theta_N(1+\epsilon) \le (1+ k^{-1}) \theta_N$ for all $N$ large
enough if $k> \epsilon^{-1}$. Taking $k= [\epsilon^{-1}] +1$, where
$[a]$ stands for the integer part of $a$, we conclude from the
previous two estimates that for $N$ large enough
\begin{equation*}
\hat \theta_N \;\le\; \Big(1+ \frac 1{[\epsilon^{-1}] +1} \Big) 
\, \theta_N  \;+\; L \, T^{\rm mix}_N
\end{equation*}
provided $2^{-L} \le e^{-1} [ 1 - e^{-\epsilon}]$. This proves that
for every $\epsilon >0$, $\limsup_N (\hat \theta_N/\theta_N)\le 1+
([\epsilon^{-1}] +1)^{-1}$, i.e., that $\limsup_N (\hat
\theta_N/\theta_N)\le 1$, proving claim \eqref{04}.

It follows from Lemma \ref{s03} and \eqref{04} that $H_N/\hat
\theta_N$ converges in distribution to a mean one exponential random
variable.  We claim that
\begin{equation}
\label{03}
\lim_{N\to\infty} \frac{\bb E_{\nu_N} [\,
  H_N\,]}{\hat\theta_N} \;=\; 1\;.
\end{equation}
To prove \eqref{03}, we change variables to obtain that
\begin{equation*}
\hat\theta_N^{-1} \bb E_{\nu_N} [\,H_N\,] \;=\;
\hat\theta_N^{-1} \int_0^\infty  \bb P_{\nu_N} 
[\,H_N > t \,]\, dt \;=\;
\int_0^\infty  \bb P_{\nu_N} 
[\,H_N / \hat\theta_N  > t \,]\, dt \;.
\end{equation*}
It remains to obtain a bound to apply the dominated convergence
theorem. By definition of $\hat\theta_N$, $\bb P_{\eta} [\,H_N >
\hat\theta_N\,] \le e^{-1}$ for all $\eta\in \Omega_N$. By the Markov
property, we obtain that $\bb P_{\eta} [\,H_N > t \hat\theta_N\,] \le
\bb P_{\eta} [\,H_N > [t] \hat\theta_N\,] \le e^{-[t]}$ so that $\bb
P_{\nu_N} [\,H_N / \hat\theta_N > t \,]\le e^{-[t]}$.
\end{proof}

\begin{corollary}
\label{s13}
Assume that the hypotheses of Theorem \ref{s04} are fulfilled. Let
$\{\mu_N : N\ge 1\}$ be a sequence of probability measures and suppose
that there exists a sequence $S_N$, $T^{\rm mix}_N \ll S_N \ll \bb
E_{\nu_N} [H_N]$, such that
\begin{equation}
\label{19}
\lim_{N\to\infty} \bb P_{\mu_N} [\, H_N < S_N \,] \;=\; 0\;.
\end{equation}
Then, under $\mu_N$, $H_N/ \bb E_{\nu_N} [H_N]$ converges in
distribution to a mean one exponential random variable.
\end{corollary}

\begin{proof}
Let $\bb U_N = \bb E_{\nu_N} [H_N]$ and fix  $t>0$. Clearly,
\begin{equation*}
\bb P_{\mu_N} [\, H_N \le t \, \bb U_N \,] \;=\; \bb P_{\mu_N} [\, S_N \le
H_N \le t \, \bb U_N \,] \;+\; \bb P_{\mu_N} [\, S_N > H_N \,,\, 
H_N \le t \, \bb U_N \,]\;.
\end{equation*}
By assumption, the second term on the right hand side vanishes as
$N\uparrow\infty$, while the first one, by the Markov property, is
equal to
\begin{equation*}
\begin{split}
& \bb E_{\mu_N} \big[\, \mb 1\{S_N \le H_N\}\, 
\bb P_{\eta(S_N)}  [\, H_N \le t \, \bb U_N - S_N \,] \,\big] \\
&\quad =\; \bb E_{\mu_N} \big[\, 
\bb P_{\eta(S_N)}  [\, H_N \le t \, \bb U_N - S_N \,] \,\big] \\
&\qquad -\; \bb E_{\mu_N} \big[\, \mb 1\{S_N > H_N\}\, 
\bb P_{\eta(S_N)}  [\, H_N \le t \, \bb U_N - S_N \,] \,\big]
\;.   
\end{split}
\end{equation*}
As before, the second term on right hand side vanishes as
$N\uparrow\infty$. The first one, since $T^{\rm mix}_N \ll S_N$ is
equal to 
\begin{equation*}
\bb P_{\nu_N}  [\, H_N \le t \, \bb U_N - S_N \,] \;+\; R_N\;,
\end{equation*}
where $\lim_N R_N =0$. Since $S_N\ll \bb U_N$, by Theorem \ref{s04},
the first term in the previous displayed formula converges to
$1-e^{-t}$, which proves the corollary.
\end{proof}

To apply the previous corollary one needs among other things to
estimate $\bb P_{\mu_N} [\, H_N < S_N \,]$ and $\bb E_{\nu_N}
[H_N]$. In the next section we present a general method to estimate
the latter sequence when a dynamical large deviations principle is
available. There are several ways to bound $\bb P_{\mu_N} [\, H_N <
S_N \,]$. We present below three approaches. The first two uses the
enlarged processes introduced by Bianchi and Gaudilli\`ere \cite{bg1},
the second and the third ones are taken from the martingale approach
to metastability \cite{bl9}.

Consider a sequence $\gamma_N$ of positive real numbers. Let
$\Omega^\star_N$ be a copy of the set $\Omega_N$ and recall from
\cite[Section 2.C]{bl9} the definition of the enlarged process
associated to the sequence $\gamma_N$, a Markov process, denoted by
$\eta^\star(t)$, on $\Omega_N \cup \Omega^\star_N$ which jumps from a
state $\eta\in \Omega_N$ to its copy $\eta^\star\in \Omega^\star_N$ at
rate $\gamma_N$. Denote by $\nu^\star_N$ the stationary measure of the
enlarged process and recall that $\nu^\star_N (\eta) = \nu^\star_N
(\eta^\star) = (1/2) \nu_N (\eta)$. Let $\Cap_\star$ be the capacity
with respect to the enlarged process, and for a subset $B$ of
$\Omega_N$, denote by $B^\star$ the copy of the set $B$.  Next result
is Corollary 4.2 in \cite{bl9}.

\begin{lemma}
\label{s17}
Let $\mu_N$ be a sequence of probability measures concentrated on
$A^c_N$ and set $\gamma_N = S^{-1}_N$. Assume that
\begin{equation*}
\lim_{N\to\infty} S_N  \, E_{\nu_N} \Big[ \Big( 
\frac{d\mu_N}{d\nu_N}\Big)^2 \Big]\, 
\Cap_\star(A_N, (A^{c}_N)^\star)\;=\; 0\;.
\end{equation*}
Then, \eqref{19} holds.
\end{lemma}

Theorems 2.4 and 2.7 in \cite{gl2} provide variational formulae for
the capacity. The second theorem expresses the capacity as an infimum
over flows. It permits, in particular, to obtain simple upper bounds.
An elementary bound for the capacity is obtained as follows. By
definition of the capacity and since $\nu^\star (\eta) = (1/2)
\nu(\eta)$,
\begin{equation*}
\Cap_\star(A_N, (A^{c}_N)^\star) \;=\; (1/2) \sum_{\eta\in A_N}
\nu(\eta) \big\{ R_N(\eta, \Omega_N) + \gamma_N \big\}\,
\bb P^\star_\eta \big[H_{(A^{c}_N)^\star} < H^+_{A_N} \big] \;,
\end{equation*}
where $\bb P^\star_\eta$ represent the distribution of the enlarged
process $\eta^\star(t)$ starting from $\eta$. Therefore, \eqref{19}
holds if
\begin{equation}
\label{22}
\lim_{N\to\infty} S_N  \, E_{\nu_N} \Big[ \Big( 
\frac{d\mu_N}{d\nu_N}\Big)^2 \Big]\, 
\big\{ S^{-1}_N + \max_{\eta\in A_N} R_N(\eta, \Omega_N)\big\} 
\, \nu_N(A_N) \;=\; 0\;.
\end{equation}

\begin{lemma}
\label{s14}
Let $\mu_N$ be a sequence of probability measures on $\Omega_N$ and
let $R'(\eta, A_N) = \mb 1\{\eta\in A^c_N\}\, R_N(\eta, A_N)$. Assume
that
\begin{equation*}
\lim_{N\to\infty} \Big\{ \mu_N(A_N) \;+\; E_{\nu_N} \big[R'_N(\eta, A_N)\big]\,
\sum_{\eta\in \Omega_N} \mu_N(\eta) \,
\frac{1} {\Cap_\star (\eta, \Omega^\star_N)} \Big\}\;=\; 0
\end{equation*}
for some sequence $\gamma^{-1}_N \gg S_N$. Then, \eqref{19} holds.
\end{lemma}

In the reversible case, the Thomson principle permits to estimate from
below the capacity.  If $\gamma^{-1}_N \gg T^{\rm mix}_N$, starting
from any state in $\Omega_N$, the distribution of
$\eta^\star(H_{\Omega^\star_N})$ is close to the stationary state
$\nu_N$ lifted to $\Omega^\star_N$. Since the capacity can be
interpreted as the inverse of a distance, the sum on the right hand
side measures the distance from $\mu_N$ to the stationary state
$\nu_N$.

\begin{proof}[Proof of Lemma \ref{s14}]
We first replace the deterministic sequence $S_N$ in \eqref{19} by a
sequence of exponential random variables independent of the Markov
process $\eta(t)$. Denote by $\mf e_N$ a mean $\gamma^{-1}_N$
exponential time independent of the Markov process $\eta(t)$. Since
$S_N \ll \gamma^{-1}_N$, 
\begin{equation}
\label{21}
\limsup_{N\to\infty} \bb P _{\mu_N} [\, H_N < S_N \,] \;\le\;
\liminf_{N\to\infty} \bb P _{\mu_N} [\, H_N < \mf e_N \,]\;.
\end{equation}

Repeating the steps which led to \eqref{01}, we obtain that
\begin{equation*}
\bb P_{\mu_N} [\, H_N < \mf e_N \,] \;\le\; \mu_N(A_N) \;+\;
\bb E_{\mu_N} \Big[\, \int_0^{\mf e_N} R'_N(\eta(s), A_N) \, ds \, \Big
] \;.
\end{equation*}
In this step we used twice the monotone convergence theorem and we
replaced $\mf e_N$ by $\mf e_N \wedge t$ to overcome the unboundedness
of $\mf e_N$.

Clearly, starting from any configuration in $\Omega_N$, we may
interpret $\mf e_N$ as the hitting time of $\Omega^\star_N$ for the
enlarged process so that
\begin{equation*}
\bb E_{\mu_N} \Big[\, \int_0^{\mf e_N} R'_N(\eta(s), A_N) \, ds \,
\Big] \;=\; \bb E^\star_{\mu_N} \Big[\, \int_0^{H_{\Omega^\star_N}} 
R'_N(\eta^\star(s), A_N) \, ds \, \Big] \;.
\end{equation*}
By \cite[Proposition A.2]{bl7}, since the equilibrium potential is
bounded by $1$ and since $\nu^\star_N (\eta) = (1/2) \nu_N(\eta)$,
$\eta\in\Omega_N$, the previous expectation is equal to
\begin{equation*}
\begin{split}
& \sum_{\eta\in \Omega_N} \mu_N(\eta) \, \bb E^\star_{\eta}
\Big[\, \int_0^{H_{\Omega^\star_N}} 
R'_N(\eta^\star(s), A_N) \, ds \, \Big] \\
&\quad \;\le \; E_{\nu_N} \big[R'_N(\eta, A_N)\big]\,
\sum_{\eta\in \Omega_N} \mu_N(\eta) \,
\frac{1} {2\, \Cap_\star (\eta, \Omega^\star_N)}\;,
\end{split}
\end{equation*}
which proves the lemma.
\end{proof}

We conclude this section with a third set of sufficient conditions
for \eqref{19}. Denote by $T^{\rm rel}_N$ the relaxation time, i.e.
the inverse of the spectral gap of the symmetric part of the
generator, and denote by $\Vert\,\cdot\,\Vert_p$ the norm of $L^p(\nu_N)$,
$0<p\le \infty$.

\begin{lemma}
\label{s15}
Let $S_N$ be an increasing sequence and let $\mu_N$ be a sequence of
probability measures on $\Omega_N$. Assume that
\begin{equation*}
\begin{split}
& \lim_{N\to\infty} \Big\{ \mu_N(A_N) \;+\; S_N\, 
\nu_N (A_N^c)\, r_N(A_N^c, A_N) \Big\}\;=\; 0\;, \\
&\quad  
\lim_{N\to\infty} \Vert R^\prime_N( \,\cdot\,, A_N) \Vert_2 \, 
\Big\Vert \frac{d\mu_N}{d\nu_N} \Big\Vert_{2} 
\, T^{\rm rel}_N\, \Big( 1 - e^{-S_N/T^{\rm rel}_N} \Big) \;=\; 0\;,
\end{split}
\end{equation*}
where $R^\prime_N(\eta, A_N) = \mb 1\{\eta \not\in A_N\} \, R_N(\eta,
A_N)$. Then, \eqref{19} holds.
\end{lemma}

We may estimate $\Vert R^\prime_N( \,\cdot\,, A_N) \Vert^2_2$ by
$\Vert R^\prime_N( \,\cdot\,, A_N) \Vert_\infty \, \Vert R^\prime_N(
\,\cdot\,, A_N) \Vert_1$ and recall that $\Vert R^\prime_N( \,\cdot\,,
A_N) \Vert_1 = \nu_N (A_N^c)\, r_N(A_N^c, A_N)$ which vanishes
asymptotically.

\begin{proof}[Proof of Lemma \ref{s15}]
Repeating the steps which led to \eqref{01}, we
obtain that
\begin{equation*}
\begin{split}
& \bb P_{\mu_N} [\, H_N < S_N \,] \;\le\; \mu_N(A_N) \;+\;
\bb E_{\mu_N} \Big[\, \int_0^{S_N} \mb 1\{\eta(s) \not\in A_N\} \,
R_N(\eta(s), A_N) \, ds \, \Big ]
\\
&\qquad =\;
\mu_N(A_N) \;+\; S_N\, \nu_N (A_N^c)\, r_N(A_N^c, A_N) \;+\;
\int_0^{S_N} \bb E_{\mu_N} \Big[\,\hat R_N(\eta(s)) \, \Big ] \, ds
\; ,
\end{split}
\end{equation*}
where $\hat R_N(\eta)$ is the $\nu_N$-mean zero function $\hat
R_N(\eta) = R^\prime_N(\eta, A_N) - E_{\nu_N}[R^\prime_N(\eta, A_N)]$,
and $E_{\nu_N}[R^\prime_N(\eta, A_N)] = \nu_N (A_N^c)\, r_N(A_N^c,
A_N)$.

We estimate the last term of the previous displayed equation.  Let
$f_s (\eta)$, $\eta\in\Omega_N$, $s\ge 0$, be the unique solution of
\begin{equation*}
f_0(\eta)\;=\; \frac{\mu_N(\eta)}{\nu_N(\eta)}\;, \quad \frac d{ds}
f_s = L^*_N f_s\;,
\end{equation*}
where $L^*_N$ stands for the adjoint of $L_N$ in $L^2(\nu_N)$. With
this notation the integral in the penultimate displayed equation
becomes
\begin{equation*}
\int_0^{S_N} \<\hat R_N , f_s \>_{\nu_N} \, ds \;\le\; 
\Vert \hat R_N \Vert_2 \,
\int_0^{S_N} \< f_s ; f_s \>_{\nu_N}^{1/2} \, ds \;,
\end{equation*}
where $\<\,\cdot\,,\,\cdot\, \>_{\nu_N}$ represents the scalar product
in $L^2(\nu_N)$ and $\<f_s; f_s \>_{\nu_N}$ the variance of $f_s$. It
is well known that $\< f_s ; f_s \>_{\nu_N} \le \< f_0 ; f_0
\>_{\nu_N} e^{-2s/T^{\rm rel}_N}$. The previous expression is thus
bounded by
\begin{equation*}
\Vert \hat R_N \Vert_2 \, \< f_0 ; f_0 \>_{\nu_N}^{1/2} 
\, T^{\rm rel}_N\, \Big( 1 - e^{-S_N/T^{\rm rel}_N} \Big)\;,
\end{equation*}
which proves the lemma by replacing variances by $L^2$ norms.
\end{proof}

\section{Expectations of hitting times}
\label{sec3}

We showed in the previous section that in the context of finite state
Markov processes, the hitting time of rare events is asymptotically
distributed according to an exponential law. We show in this section
that the expectation under the stationary measure of these hitting
times can be estimated if one is able to prove a dynamical large
deviations principle. Instead of presenting this result in a general
setting, we examine the case of the BDSSEP.

\subsection*{The dynamical large deviation principle}

We recall a result first proved in \cite{bdgjl3}, and then in
\cite{blm1} in the form presented below. We say that sequence of
configurations $\{\eta^N : N\ge 1\}$, $\eta^N\in \Omega_N$, is
\emph{associated} to the macroscopic density profile $\rho\in\ms M$ if
the sequence $\pi^N(\eta^N)$ converges to $\rho$ in $\ms M$ as
$N\to\infty$.

Given $T>0$, we denote by $D\big([0,T];\ms M\big)$ the Skorohod space
of paths from $[0,T]$ to $\ms M$ equipped with its Borel
$\sigma$-algebra. Elements of $D\big([0,T], \ms M\big)$ will be
denoted by $u(t)$ and sometimes $u_t$.  

Fix a profile $\gamma\in \ms M$ and consider a sequence $\{\eta^N :
N\ge 1\}$ associated to $\gamma$. It has been proven in
\cite{els2,klo2} following the work of \cite{dv, kov} that as
$N\to\infty$ the sequence of random variables
\begin{equation}
\label{13}
\pi^N(t) \;:=\; \pi^N(\eta^N (t N^2))\;,
\end{equation}
which take values in $D\big([0,T], \ms M\big)$, converges in
probability to the unique weak solution $u(t)$ of the heat equation
with Dirichlet boundary conditions:
\begin{equation}
\label{09}
\left\{
\begin{array}{l}
\partial_t u = (1/2) \Delta u \;, \\
u(t,0) = \alpha\;, \quad u(t,1) = \beta \;, \quad t\ge 0\;, \\
u(0,x) = \gamma(x) \;, \quad 0\le x\le 1 \;.
\end{array}
\right. 
\end{equation}
Note that time has been speeded-up by $N^2$ in \eqref{13}.  

Recall the definition of the rate functional $I_{[0,T]} (\cdot |
\gamma)$ of the dynamical large deviations principle introduced in
\eqref{26}. The next two results have been proven in \cite{blm1}.

\begin{lemma}
\label{s11}
Fix $\gamma\in \ms M$ and $T>0$.  The functional $I_{[0,T]} (\cdot |
\gamma)$ is lower semicontinuous and has compact level sets. Any path
$u$ with finite rate function, $I_{[0,T]} (u | \gamma)<\infty$, is
continuous in time and satisfies the boundary conditions $u(0,\cdot) =
\gamma(\cdot)$, $u(\cdot, 0) = \alpha$, $u(\cdot, 1) =
\beta$. Furthermore, any trajectory $u$ with finite rate function can
be approximated by a sequence of smooth trajectories $\{u^n : n\ge
1\}$ in such a way that $I_{[0,T]}(u^n | \gamma)$ converges to
$I_{[0,T]}(u | \gamma)$.
\end{lemma}

The dynamical large deviation principle can now be stated.

\begin{theorem}
\label{s08}
Fix $T>0$ and an initial profile $\gamma$ in $\ms M$.  Consider a
sequence $\{\eta^N : N\ge 1\}$ of configurations associated to
$\gamma$.  Then, the sequence of probability measures $\{\bb
P^N_{\eta^N} \circ (\pi^N (N^2 \,\cdot\,)^{-1} : N\ge 1\}$ on
$D([0,T], \ms M)$ satisfies a large deviation principle with speed $N$
and good rate function $I_{[0,T]}(\cdot|\gamma)$.  Namely,
$I_{[0,T]}(\cdot|\gamma): D\big([0,T]; \ms M\big) \to [0,\infty]$ has
compact level sets and for each closed set $\mf C \subset D([0,T], \ms
M)$ and each open set $\mf O \subset D([0,T], \ms M)$
\begin{eqnarray*}
&& 
\limsup_{N\to\infty} \frac 1N \log \bb P^N_{\eta^N} 
\big( \pi^N (N^2 \,\cdot\,) \in \mf C\big)
\;\leq\; - \inf_{u \in \mf C} I_{[0,T]} (u | \gamma)  
\\
&& \qquad 
\liminf_{N\to\infty} \frac 1N \log \bb P^N_{\eta^N} 
\big( \pi^N (N^2 \,\cdot\,) \in \mf O \big) 
\;\geq\; -  \inf_{u \in \mf O} I_{[0,T]} (u |\gamma) \; . 
\end{eqnarray*}
\end{theorem}

\subsection*{The static large deviation principle}
The large deviations principle for the empirical measure under the
stationary state $\nu^N_{\alpha, \beta}$, stated below, is taken from
\cite{bg2, f1}.

\begin{theorem}
\label{s09}
The sequence of probability measures $\{\nu^N_{\alpha, \beta} \circ
(\pi^N)^{-1} : N\ge 1\}$ on $\ms M$ satisfies a large deviation
principle with speed $N$ and good rate function $V$.  Namely, $V: \ms
M \to [0,\infty]$ has compact level sets and for each closed set $\ms
C \subset \ms M$ and each open set $\ms O \subset \ms M$
\begin{eqnarray*}
&& 
\limsup_{N\to\infty} \frac 1N \log \nu^N_{\alpha, \beta}
\big( \pi^N \in \ms C\big)
\;\leq\; - \inf_{\gamma \in \ms C} V(\gamma)  
\\
&& \qquad 
\liminf_{N\to\infty} \frac 1N \log \nu^N_{\alpha, \beta}
\big( \pi^N\in \ms O \big) \;\geq\; -  \inf_{\gamma \in \ms O} 
V(\gamma) \; . 
\end{eqnarray*}
\end{theorem}

\subsection*{Expectation of hitting times}

The main result of this section can now be stated. Fix an open subset
$\ms O$ of $\ms M$ and let 
\begin{equation*}
A_N \;=\; (\pi^N)^{-1}(\ms O) = \{\eta\in\Omega_N : 
\pi^N(\eta)\in \ms O\}\;,
\end{equation*}
and let $H_N = H_{A_N}$ be the hitting time of the set $A_N$. Note
that $H_N$ coincides with the hitting time $H_{\ms O}$ introduced in
Theorem \ref{s16}.

\begin{theorem}
\label{s06}
Fix an open subset $\ms O$ of $\ms M$. Assume that $T^{\rm mix}_N
\ll \exp\{aN\}$ for all $a>0$, that $H_N/\bb E_{\nu^N_{\alpha, \beta}}
[H_N]$ converges in distribution to a mean one exponential random
variable, and that
\begin{equation*}
V(\ms O) \;:=\; \inf_{\gamma \in \ms O} V(\gamma) 
\;=\; \inf_{\gamma \in \overline{\ms O}} V(\gamma)\;.
\end{equation*}
Then, for every $\epsilon>0$,
\begin{equation*}
\liminf_{N\to\infty} \frac {\bb E_{\nu^N_{\alpha, \beta}} [H_N]}
{e^{N \{ V(\ms O) - \epsilon\}}} \;>\; 0\;, \quad 
\limsup_{N\to\infty} 
\frac {\bb E_{\nu^N_{\alpha, \beta}}[H_N]}
{e^{N [V(\ms O) + \epsilon]}}
\;<\; \infty \;.
\end{equation*}
In particular,
\begin{equation*}
\lim_{N\to\infty} \frac 1N \log 
\bb E_{\nu^N_{\alpha, \beta}}[H_N] \;=\; V(\ms O)\;.
\end{equation*}
\end{theorem}

To prove this result we first need a dynamical large deviations
principle starting from the stationary measure. 

\begin{theorem}
\label{s10}
For each $T>0$, each closed set $\mf C \subset D([0,T], \ms
M)$ and each open set $\mf O \subset D([0,T], \ms M)$,
\begin{eqnarray*}
&& 
\limsup_{N\to\infty} \frac 1N \log \bb P^N_{\nu^N_{\alpha, \beta}} 
\big( \pi^N (N^2 \,\cdot\,)\in \mf C\big)
\;\leq\; - \inf_{u \in \mf C} \big\{ I_{[0,T]} (u | u_0)  + V(u_0) \big\}
\\
&& \qquad 
\liminf_{N\to\infty} \frac 1N \log \bb P^N_{\nu^N_{\alpha, \beta}} 
\big( \pi^N (N^2 \,\cdot\,) \in \mf O \big) \;\geq\; -  \inf_{u \in \mf O} 
\big\{ I_{[0,T]} (u | u_0)  + V(u_0) \big\}\; . 
\end{eqnarray*}
\end{theorem}

\begin{proof}
In order to simplify the expressions, we will use the fact that
concerning the SSEP process, as mentioned in \cite[last part of
section 2]{blm1}, the two dynamical rate functionals
$I_{[0,T]}(u|\gamma)$ and $\hat I_{[0,T]}(u|\gamma)$ (see \eqref{26})
are the same.

We start with the proof of the upper bound.  The arguments closely
follow the ones presented in \cite{blm1}. Theorem \ref{s09} is used
afterwards to estimate the large deviations from the initial
stationary distribution.

It is well known that using an exponential tightness argument, it is
enough to prove the upper bound for compact sets.  For any function
$(t,x)\mapsto H_t(x)\in C_0^{1,2}([0,T]\times[0,1])$, we introduce the
exponential martingale $M^H_t$ defined by
\[
\begin{split}
M^H_t \;=\; \exp \Big\{ N \Big[ \<\pi^N_t, H_t\> 
&- \<\pi^N_0, H_0\> \\
&- \frac 1N \int_0^t e^{- N \<\pi^N_s, H_s\>} (\partial_s + L_N) 
\, e^{N\<\pi^N_s, H_s\>} \,ds \Big] \Big\}\; .  
\end{split}
\]
Using a super-exponential estimate (\cite[Theorem
3.2]{blm1}), for any $\delta>0$ and $\epsilon>0$, there exists a set of configurations
$\eta\in B^{H,N}_{\delta , \epsilon}$ such that for any $\delta>0$
\[
\lim_{\epsilon\to 0}\,\limsup_{N\to\infty} \frac 1N  \log
\bb  P^N_{\nu^N_{\alpha, \beta}} \Big [ \big(B^{H,N}_{\delta , \epsilon} \big)^\complement\Big] \; 
=\; -\infty \; .
\]
and on which
\begin{equation*}
M^H_T \;=\;  \exp N\Big\{  \hat J_{H} (\pi^{N,\epsilon}|\pi^N_0) 
\; +\; O_H (\epsilon) \;+\; O(\delta) \Big\} \;,
\end{equation*}
where the functional $\hat J_H$ was defined in \eqref{o01}, $O_{H}
(\epsilon)$ (resp. $O(\delta)$) is an deterministic expression which vanishes as
$\epsilon \downarrow 0$ (resp. $\delta \downarrow 0$) and where, for
any density $\pi\in \ms M$,
\[
\pi^\epsilon(u)=
\frac 1{2\epsilon}\int_{[u-\epsilon,u+\epsilon]\cap[0,1]}\!\!
\pi(u)\,du.
\]
Let $\ms K$ be a compact subset of $D([0,T], \ms M)$, then
\[
\begin{split}
 \limsup_{N\to\infty} \frac 1N  \log
\bb  P^N_{\nu^N_{\alpha, \beta}}
&\big[\pi^N\in\ms K\big]
\\
\le 
&\limsup_{\delta\downarrow 0}\limsup_{\epsilon\downarrow 0}
\limsup_{N\to\infty} \frac 1N  \log
\bb  P^N_{\nu^N_{\alpha, \beta}}
\big[\{\pi^N\in\ms K\}\cap B^{H,N}_{\delta , \epsilon}\big]  
\end{split}
\]
and we can write
\[
\bb P^N_{\nu^N_{\alpha, \beta}}
\big[\{\pi^N\in\ms K\}\cap B^{H,N}_{\delta , \epsilon}\big]
=\bb E^N_{\nu^N_{\alpha, \beta}}
\big[M^H_T (M^H_T )^{-1}\mb 1_{\{\pi^N\in\ms K\}\cap B^{H,N}_{\delta , \epsilon}}\big].
\]
Therefore,
\[
\begin{split}
 \frac 1N  \log\bb P^N_{\nu^N_{\alpha, \beta}}
&\big[\{\pi^N\in\ms K\}\cap B^{H,N}_{\delta , \epsilon}\big]\\
&\le \frac 1N  \log\bb E^N_{\nu^N_{\alpha, \beta}}
\Big[M^H_T  \exp N\sup_{u\in \ms K}\big\{  -\hat J_{H}
(u^\epsilon|\pi^N_0)\big \} \Big]+O_H(\epsilon)+O(\delta)
\end{split}
\]
and since $M^H_T$ is a mean $1$ martingale, we get
\[
\begin{split}
 \limsup_{N\to\infty} \frac 1N  \log\;
&\bb  P^N_{\nu^N_{\alpha, \beta}}
\big[\pi^N\in\ms K\big]
\\
\le 
&\limsup_{\epsilon\downarrow 0}
\limsup_{N\to\infty} \frac 1N  \log
E_{\nu^N_{\alpha, \beta}}
\Big[\exp N\sup_{u\in \ms K}\big\{  -\hat J_{H}
(u^\epsilon|\pi^N)\big \} \Big].
\end{split}
\]
We notice that the map $\pi\mapsto \sup_{u\in \ms K}\{ -\hat
J_{H} (u^\epsilon|\pi)\}$ is continuous on $\ms M$, so we can
apply Varadhan's Lemma to the large deviation principle stated in Theorem
\ref{s09}
\[
\begin{split}  
\lim_{N\to\infty} \frac 1N  
&\log
E_{\nu^N_{\alpha, \beta}}
\Big[\exp N\sup_{u\in \ms K}\big\{  -\hat J_{H}
(u^\epsilon|\pi^N)\big \} \Big]\\
&=\sup_{\gamma\in\ms M}\big\{\sup_{u\in \ms K}\{  -\hat J_{H}
(u^\epsilon|\gamma)\}-V(\gamma)\big\}
=-\inf_{\gamma\in\ms M, u\in\ms K}\{  \hat J_{H}
(u^\epsilon|\gamma)+V(\gamma)\}.
\end{split}
\]
Now, since $\ms M\times\ms K$ is compact, we can follow step by step
the arguments of \cite[section 3.3]{blm1} and we get
\[
 \limsup_{N\to\infty} \frac 1N  \log
\bb  P^N_{\nu^N_{\alpha, \beta}}
\big[\pi^N\in\ms K\big]
\le -\inf_{\gamma\in\ms M, u\in\ms K}\{I_{[0,T]}(u,\gamma)+V(\gamma)\},
\]
which is precisely the required upper bound since
$I_{[0,T]}(u,\gamma)<+\infty$ implies that $u_0=\gamma$.

The proof of the lower bound is easier. Indeed recalling the
definition of the rate function $V$, we only have to show that for any
$u\in D([0,T],\ms M)$, any $S>0$, any $\pi\in D([-S,0],\ms M)$ such that
$\pi_{-S}=\bar \rho$ and $\pi_{0}=u_0$, and for any $\delta>0$,
\[
 \liminf_{N\to\infty} \frac 1N  \log
\bb  P^N_{\nu^N_{\alpha, \beta}}
\big[\pi^N\in B_{[0,T]}(u,\delta)\big]
\ge -I_{[-S,0]}(\pi|\bar \rho)-I_{[0,T]}(u|u_0),
\]
where $B_{[0,T]}(u,\delta)$ is the ball centered at $u$ with radius
$\delta$ for the Skorohod topology on $D([0,T],\ms M)$.
If we denote by $\tilde u$ the density path given by
$\pi$ on $[-S,0]$ and $u$ on $[0,T]$, then $\tilde u\in D([-S,T],\ms
M)$ and $I_{[-S,T]}(\tilde u)=I_{[-S,0]}(\pi|\bar
\rho)+I_{[0,T]}(u|u_0)$.
Therefore, since $\nu^N_{\alpha, \beta}$ is a stationary
distribution, we have 
\[
\bb P^N_{\nu^N_{\alpha, \beta}}
\big[\pi^N\in B_{[0,T]}(u,\delta)\big]
\ge \bb P^N_{\nu^N_{\alpha, \beta}}
\big[\pi^N\in B_{[-S,T]}(\tilde u,\delta)\big].
\]
As under $\nu^N_{\alpha, \beta}$ the initial empirical density
$\pi^N_0$ converges to the stationary density $\bar\rho$, the lower
bound proved in \cite{blm1} applies here and we get
\[
 \liminf_{N\to\infty} \frac 1N  \log
\bb  P^N_{\nu^N_{\alpha, \beta}}
\big[\pi^N\in B_{[-S,T]}(u,\delta)\big]
\ge -I_{[-S,T]}(\tilde u|\bar\rho).
\]
\end{proof}

Next lemma is also needed in the proof of Theorem \ref{s06}.

\begin{lemma}
\label{s07}
Fix a subset $\ms B$ of $\ms M$ and $T>0$. Let $\mf A = \{u \in
C([0,T], \ms M): u(t) \in \ms B$ for some $0\le t\le T\}$. Then,
\begin{equation*}
\inf_{u\in \mf A} \big\{ I_{[0,T]} (u | u_0) + V(u_0)\} \;\ge\;
\inf_{\rho\in\ms B} V(\rho)\;.
\end{equation*}
\end{lemma}

\begin{proof}
Fix $\epsilon >0$ and $u\in \mf A$. Assume that $u(t_0) \in\ms B$,
$0\le t_0 \le T$. By \eqref{08}, there exists $T_0>0$ and a path $v\in
C([-T_0, 0], \ms M)$ such that $v(-T_0) = \bar\rho$, $v(0) = u(0) = u_0$,
$I_{[-T_0, 0]}(v|\bar\rho) \le V(u_0) + \epsilon$. Defining the path
$w$ in $C [-T_0, t_0], \ms M)$ by $w(t) = v(t)$, $-T_0 \le t\le 0$,
$w(t) = u(t)$, $0 \le t\le t_0$, we obtain a path connecting
$\bar\rho$ to $u(t_0) \in\ms B$. By \eqref{08}, $I_{[-T_0,t_0]}
(w|\bar\rho) \ge \inf_{\rho\in\ms B} V(\rho)$. It follows from the
estimates just obtained that
\begin{equation*}
\begin{split}
& I_{[0,T]} (u | u_0) + V(u_0) \;\ge\; I_{[0,t_0]} (u | u_0) + V(u_0) \\
& \quad \;=\; I_{[-T_0,t_0]} (w|\bar\rho) + V(u_0) - I_{[-T_0,0]} (v|\bar\rho)
\;\ge\; \inf_{\rho\in\ms B} V(\rho) \;-\; \epsilon\;,  
\end{split}
\end{equation*}
which proves the lemma.
\end{proof}

\begin{proof}[Proof of Theorem \ref{s06}]
Fix $\epsilon>0$. There exists $\gamma\in \ms O$ such that
\begin{equation*}
V(\gamma) \;<\; \inf_{\rho\in\ms O} V(\rho) \;+\; (\epsilon/2) \;,
\end{equation*}
and there exists $\delta$ such that $B_\delta(\gamma) \subset \ms O$.
By \eqref{08} and by translation invariance of the dynamical rate
function, there exist $T_\epsilon>0$ and a path $u^{(\epsilon)} (t)$,
$0\le t\le T_\epsilon $, $u^{(\epsilon)}_0 = \bar\rho$,
$u^{(\epsilon)}(T_\epsilon)=\gamma$ such that
\begin{equation}
\label{11}
I_{[0, T_\epsilon]} (u^{(\epsilon)} | \bar\rho ) \;<\; \inf_{\rho\in\ms O} V(\rho) 
\;+\; \epsilon\;.
\end{equation}

For $\varphi>0$, $T>0$ and a path $u\in D([0,T], \ms M)$ denote by
$\bb B_{\varphi, T}(u)$ the open ball in $D([0,T], \ms M)$ of radius
$\varphi$ centered around $u$. Let $G = \bb B_{\delta,
  T_\epsilon}(u^{(\epsilon)})$, $G_{L} = \{ u \in D([0, L \, T^{\rm mix}_N
/ N^2 + T_\epsilon], \ms M) : u(L \, T^{\rm mix}_N /N^2 +
\,\cdot\,) \in G\}$ It is clear from the definition of $G$ that $G_{L}
\subset \{H_N \le L \, T^{\rm mix}_N + T_\epsilon N^2\}$. Hence, for
any configuration $\xi\in\Omega_N$,
\begin{equation*}
\bb P_\xi \big[ H_N \le L \, T^{\rm mix}_N + T_\epsilon N^2 \big]
\;\ge\; \bb P_\xi \big[ \pi^N \in G_L \big]
\;=\; \sum_{\zeta\in\Omega_N} P_{L \, T^{\rm mix}_N}(\xi,\zeta) \,
\bb P_\zeta \big[ \pi^N \in G \big]\;,
\end{equation*}
where $P_t(\eta,\xi)$, $t>0$, stands for the transition probability of
the BDSSEP.  By definition of the mixing time, the previous expression
is bounded below by
\begin{equation*}
\bb P_{\nu^N_{\alpha, \beta}} \big[ \pi^N \in G \big] \;-\; 2^{-L}\;. 
\end{equation*}
Therefore, for every $L\ge 1$,
\begin{equation}
\label{12}
\inf_{\xi \in \Omega_N}
\bb P_\xi \big[ H_N \le L \, T^{\rm mix}_N + T_\epsilon N^2 \big]
\;\ge\; \bb P_{\nu^N_{\alpha, \beta}} \big[ \pi^N \in G \big] \;-\; 2^{-L}\;. 
\end{equation}

By Theorem \ref{s10}, by definition of $G$, by \eqref{11} and since
$u^{(\epsilon)}_0 = \bar\rho$, $V(\bar\rho)=0$,
\begin{equation*}
\begin{split}
\liminf_{N\to\infty} \frac 1N \log \bb P_{\nu^N_{\alpha, \beta}} 
\big[ \pi^N \in G\big] \;& \ge\; -\inf_{u\in G} 
\big\{ I_{[0,T_\epsilon]} (u | u_0 ) + V(u_0)\big\} \\
\; & \ge\; - \, \{I_{[0, T_\epsilon]} (u^{(\epsilon)} | u^{(\epsilon)}_0) 
+ V(u^{(\epsilon)}_0) \big\}
\;\ge\; - (V(\ms O) + \epsilon)\;.    
\end{split}
\end{equation*}
Hence, there exists $N_0 = N_0(\epsilon, \delta)$ such that for all
$N\ge N_0$,
\begin{equation*}
\bb P_{\nu^N_{\alpha, \beta}} 
\big[ \pi^N \in G\big] \; \ge\; \exp - N 
\big \{ V(\ms O) + 2 \epsilon \big\}\;.
\end{equation*}

The previous estimate together with \eqref{12} for $L=\ell\, N$ gives
that for all $N\ge N_0$,
\begin{equation*}
\max_{\xi \in \Omega_N}
\bb P_\xi \big[ H_N > \ell\, N \, T^{\rm mix}_N + T_\epsilon N^2 \big]
\;\le\; 1 \;-\; e^{-N [V(\ms O) + 2 \epsilon]} \;+\; 2^{-\ell\, N}\;. 
\end{equation*}
Iterating this estimate $M$ times, gives by the Markov property that
\begin{equation*}
\max_{\xi \in \Omega_N}
\bb P_\xi \Big[ H_N > M \big\{\ell\, N \, T^{\rm mix}_N 
+ T_\epsilon N^2\big\}\, \Big]
\;\le\; \Big(1 \;-\; e^{-N [V(\ms O) + 2 \epsilon]} \;+\; 2^{-\ell\,
  N}\Big)^M \;. 
\end{equation*}
Taking $\ell$ large enough and setting $M= \exp\{N [V(\ms O) + 2
\epsilon]\}$, we conclude that
\begin{equation*}
\limsup_N \max_{\xi \in \Omega_N}
\bb P_\xi \Big[ H_N >  e^{N [V(\ms O) + 2 \epsilon]} \,
\big\{\ell\, N \, T^{\rm mix}_N + T_\epsilon N^2\big\}\, \Big]
\;<\; 1\;. 
\end{equation*}
Since, by assumption, $N \, T^{\rm mix}_N < \exp\{\epsilon N\}$ for
$N$ sufficiently large and since we assumed that $H_N/\bb
E_{\nu^N_{\alpha, \beta}}[H_N]$ converges to a mean one exponential
random variable, we have that
\begin{equation*}
\liminf_{N\to\infty} \frac{e^{N [V(\ms O) + 3 \epsilon]}}
{\bb E_{\nu^N_{\alpha, \beta}}[H_N]} \;>\; 0\;.
\end{equation*}

Conversely, for $k\ge 0$, let $\mf A_k = \{u \in D([k T_\epsilon,
(k+1) T_\epsilon], \ms M) : u(t)\in\ms O$ for some $k T_\epsilon \le
t\le (k+1) T_\epsilon\}$. By definition, for every $L\ge 1$
\begin{equation*}
\bb P_{\nu^N_{\alpha, \beta}} \big[ H_N\le L\, N^2\, T_\epsilon \big]
\;\le\; \sum_{k=0}^{L-1} \bb P_{\nu^N_{\alpha, \beta}} 
\big[ \pi^N \in \mf A_k \big] \;\le \; 
L\, \bb P_{\nu^N_{\alpha, \beta}} \big[ \pi^N
\in \overline{\mf A} \big]\;, 
\end{equation*}
where $\mf A = \mf A_0$ and $\overline{\mf A}$ stands for the closure
of $\mf A$.

By Theorem \ref{s10},
\begin{equation*}
\limsup_{N\to\infty} \frac 1N \log \bb P_{\nu^N_{\alpha, \beta}} 
\big[ \pi^N  \in \overline{\mf A} \big] \; \le\; 
-\inf_{u\in \overline{\mf A}} 
\big\{ I_{[0,T_\epsilon]} (u | u_0 ) + V(u_0)\big\}\;.
\end{equation*}
By Lemma \ref{s11}, we may restrict the supremum to paths $u$ in
$C([0, T_\epsilon], \ms M)$. In this case, $\overline{\mf A}$ is
contained on the closed set $\mf A' = \{u \in C([0, T_\epsilon], \ms M)
: u(t)\in\overline{\ms O}$ for some $0 \le t\le T_\epsilon\}$, so that
\begin{equation*}
\limsup_{N\to\infty} \frac 1N \log \bb P_{\nu^N_{\alpha, \beta}} 
\big[ \pi^N  \in \overline{\mf A} \big] \; \le\; 
-\inf_{u\in \mf A'}  \big\{ I_{[0,T_\epsilon]} (u | u_0 ) + V(u_0)\big\}\;.
\end{equation*}
By Lemma \ref{s07}, $\inf_{u\in \mf A'}  \big\{ I_{[0,T_\epsilon]} (u | u_0 )
+ V(u_0)\big\} \ge \inf_{\rho\in \overline{\ms O}} V(\rho)$ and this
latter quantity is by assumption equal to $\inf_{\rho\in \ms O}
V(\rho)$. Hence, there exists $N_0$ such that for all $N\ge N_0$, 
\begin{equation*}
\bb P_{\nu^N_{\alpha, \beta}} 
\big[ \pi^N  \in \overline{\mf A} \big] \; \le\; 
\exp - N \big\{ \inf_{\rho\in \ms O} V(\rho) - \epsilon \big\} \;.
\end{equation*}

Taking $L= (1/2) \exp N \big\{ V(\ms O) - \epsilon
\big\}$ we deduce from the previous estimates that
\begin{equation*}
\bb P_{\nu^N_{\alpha, \beta}} \big[ H_N\le (1/2)  
e^{N \{ V(\ms O) - \epsilon\}} \, N^2\, T_\epsilon \big]
\;\le\; 1/2
\end{equation*}
for $N$ sufficiently large. Since, by assumption,
$H_N/E_{\nu^N_{\alpha, \beta}} [H_N]$ converges in distribution to a
mean one exponential random variable, we conclude from this inequality
that
\begin{equation*}
\limsup_{N\to\infty} \frac{e^{N \{ V(\ms O) - 2 \epsilon\}}}{\bb
  E_{\nu^N_{\alpha, \beta}} [H_N]} \;<\; \infty\;.
\end{equation*}
\end{proof}

\section{Hitting times of rare events in BDSSEP} 
\label{sec4}

We prove in this section Theorem \ref{s16}. Denote by $R_N(\eta,\xi)$
the rate at which the BDSSEP $\eta(t)$ jumps from $\eta$ to $\xi$.
Recall from \eqref{24} the distance $d$ introduced in $\ms M$. With
this choice, by Schwarz inequality,
\begin{equation}
\label{15}
d(\gamma,\gamma') \;\le\; \Vert \gamma - \gamma'\Vert_2\;,
\end{equation}
where $\Vert \,\cdot\,\Vert_2$ stands for the $L_2$ norm.

\begin{lemma}
\label{s12}
Fix an open subset $\ms O$ of $\ms M$ such that $d(\bar\rho, \ms
O)>0$.  Denote by $A_N$ the set of configurations in $\Omega_N$ for
which $\pi^N(\eta)$ belongs to $\ms O$: $A_N = \{\eta\in\Omega_N:
\pi^N(\eta)\in \ms O\}$. Then, there exists $a>0$ such that
\begin{equation*}
r_N(A^c_N,A_N)\;\le\; e^{-aN} \quad\text{and}\quad \nu^N_{\alpha,
  \beta} (A_N) \;\le\; e^{-aN}
\end{equation*}
for $N$ sufficiently large. 
\end{lemma}

\begin{proof}
Let $\ms O_\delta$ is the closed set defined by $\ms O_\delta =
\{\gamma \in \ms M : d(\gamma, \overline{\ms O})\le\delta\}$,
$\delta>0$. We claim that there exists $\delta>0$ such that
\begin{equation}
\label{14}
\inf_{\gamma \in \ms O_\delta} V(\gamma) \;>\;0\;.
\end{equation}
Indeed, let $2\delta = d(\bar\rho, \ms O)>0$. It is clear from the
definition of $\ms O_\delta$ that $d(\bar\rho, \gamma) \ge \delta$ for
all $\gamma\in \ms O_\delta$.  On the other hand, by \cite[Theorem
A.1]{bdgjl3},
\begin{equation*}
V(\rho) \;\ge\; \int_0^1 \Big\{ \rho (\mb x) \log \frac{\rho (\mb x)}
{\bar \rho (\mb x)} + [1- \rho (\mb x)] \log \frac{[1- \rho (\mb x)]}
{[1- \bar \rho (\mb x)]} \Big\}\, d\mb x\;.
\end{equation*}
Therefore, since $0<\alpha\le\bar\rho(x)\le\beta<1$ and in view of
\eqref{15}, there exists $c_0>0$ such that for all $\gamma\in \ms
O_\delta$,
\begin{equation*}
V(\gamma) \;\ge\; c_0 \int_0^1 \big\{\gamma (\mb x) - \bar \rho (\mb x)
\big\}^2 \, d\mb x \;\ge\; c_0\, d(\gamma,\bar\rho)^2
\;\ge\; c_0\, \delta^2\;.
\end{equation*}

Denote by $\partial A_N$ the outer boundary of $A_N$:
\begin{equation*}
\partial A_N \;=\; \bigcup_{x=1}^{N-2}
\big\{\xi\not\in A_N : \sigma^{x,x+1}\xi\in A_N\} 
\bigcup_{z=1, N-1} 
\big\{\xi\not\in A_N : \sigma^{z}\xi\in A_N\}\;.
\end{equation*}
Since $\sum_{\xi\in\Omega_N} R_N(\eta,\xi) \le N$, by definition of
the average rate $r_N(A^c_N,A_N)$,
\begin{equation*}
r_N(A^c_N,A_N)\;\le\; \frac 1{\nu^N_{\alpha, \beta}(A^c_N)}\,  
N\, \nu^N_{\alpha, \beta} (\partial A_N)\; . 
\end{equation*}
It is clear that for each $\delta>0$, $\partial A_N\subset \{\eta\in
\Omega_N : \pi^N(\eta)\in \ms O_\delta\}$ for $N$ large enough. Hence,
by Theorem \ref{s09} and by \eqref{14}, there exists $a>0$ such
that
\begin{equation*}
\nu^N_{\alpha, \beta}(\partial A_N) \;\le\; \nu^N_{\alpha,
  \beta}(\pi^N \in \ms O_\delta) \;\le\; e^{-aN}
\end{equation*}
for $N$ sufficiently large. The same bound holds for $A_N$, which
proves the first part of the lemma.
\end{proof}

\subsection*{Estimation of the mixing time in the BDSSEP}

We show in this subsection by a coupling argument that
\begin{equation}
\label{20}
T^{\rm mix}_N \;\le\; (1/2) N^3\;.
\end{equation}
This bound is not sharp but sufficient for our purposes.

Assume that a coupling $(\eta_t, \xi_t)$ has been defined in the
product space $\Omega_N \times \Omega_N$. This means that both
coordinates evolve has the original BDSSEP and that the pair does not
leave the diagonal once it reaches it. We denote by $\bb P_{\eta,\xi}$
the distribution of the coupling when the initial configuration is
$(\eta,\xi)$. Denote by $H_{\bb D}$ the coupling time, the time the
process reaches the diagonal. It is well known that
\begin{equation*}
T^{\rm mix}_N \;\le\; \inf\big\{ t : \max_{\eta,\xi\in\Omega_N} 
\bb P_{\eta,\xi} [H_{\bb D} \ge t ] \le 1/4 \big\}\;.
\end{equation*}

The coupling of two copies of the BDSSEP is defined as follows.  Fix
two configurations $\eta$, $\xi$ in $\Omega_N$.  We assume that the
particles evolve according to a stirring dynamics and that particles
are created simultaneously in both coordinates at the boundary. In
particular, the coupled process has reached the diagonal when all
initial particles have left the system. Denote by $H_j$ the time the
particle initially at $j\in \Lambda_N$ leaves the system. If there are
no particles at $j$ set $H_j = 0$ and note that if $j$ is occupied by
an $\eta$-particle and a $\xi$-particle they both leave the system at
the same time due to the stirring dynamics. With this notation,
$H_{\bb D} \le \max_j H_j$ and for all $t>0$
\begin{equation*}
\bb P_{\eta,\xi} [H_{\bb D} \ge t ] \;\le\; \sum_{j\in\Lambda_N}
\bb P_{\eta,\xi} [H_{j} \ge t ]\;.
\end{equation*}
Under the stirring dynamics, the particle at $j$ performs a symmetric
random walk until it reaches the boundary. If we denote by
$H_{\dagger}$ the hitting time of the boundary, it is known that $\mb
E_j[H_{\dagger}] = (1/2) j(N-j) \le N^2/8$. The previous sum is thus
bounded by $N^3/8t$, which proves claim \eqref{20}.

\begin{proof}[Proof of Theorem \ref{s16}]
The first assertion of the proposition follows from Lemma \ref{s12},
\eqref{20} and Theorem \ref{s04}. The second one follows from Theorem
\ref{s06}.

To prove the third assertion, let $\gamma_N = N^4$ and consider the
enlarged process associated to this sequence. By \eqref{20} and by the
second assertion of the theorem, $T^{\rm mix}_N \ll \gamma^{-1}_N \ll
\bb E_{\nu^N_{\alpha, \beta}} [H_{\ms O}]$. 

Since $d\mu_N/d\nu^N_{\alpha, \beta} = \mb 1\{\eta\in B_N\}
\nu^N_{\alpha, \beta} (B_N)^{-1}$, $E_{\nu^N_{\alpha, \beta}} [
(d\mu_N/d\nu^N_{\alpha, \beta})^2] = \nu^N_{\alpha,
  \beta}(B_N)^{-1}$. Hence, as $R_N(\eta,\Omega_N) \le N$, the
expression appearing on the left hand side of \eqref{22} is bounded by
$N^5 \nu^N_{\alpha, \beta}(A_N) /\nu^N_{\alpha, \beta}(B_N)$. By
the static large deviation principle,
\begin{equation*}
\begin{split}
& \limsup \frac 1N \log \nu^N_{\alpha, \beta} (A_N) \;\le \;
\limsup \frac 1N \log \nu^N_{\alpha, \beta} (\pi^N \in
\overline{\ms O}) \;\le\; - \inf_{\gamma \in \overline{O}} V(\gamma)
\; , \\
& \quad \liminf \frac 1N \log \nu^N_{\alpha, \beta} (B_N) \;\ge \;
\liminf \frac 1N \log \nu^N_{\alpha, \beta} (\pi^N \in
\ms B^o) \;\ge\; - \inf_{\gamma \in \ms B^o} V(\gamma)\;.
\end{split}
\end{equation*}
Therefore, by assumption \eqref{23}, $N^5 \nu^N_{\alpha, \beta}(A_N)
/\nu^N_{\alpha, \beta}(B_N)$ vanishes as $N\uparrow\infty$. By remark
\eqref{22}, condition \eqref{19} is fulfilled. By Lemma \ref{s12} and
by \eqref{20}, the assumptions of Theorem \ref{s04} are in force. The
third assertion of the theorem follows therefore from Corollary
\ref{s13}.
\end{proof}

\smallskip\noindent{\bf Acknowledgments}. This problem was formulated
to the second author by A. Galves after a talk on nonequilibrium
stationary states at NUMEC-USP.  The authors would like to thank
A. Asselah and E. Scoppola for fruitfull discussions.

\end{document}